\documentclass{amsart}

\usepackage{amsmath}
\usepackage{amssymb}
\usepackage{amsfonts}
\usepackage{enumerate}
\usepackage{vmargin}
\usepackage{stmaryrd}

\newcommand{\tens}{\otimes}
\newcommand{\bC}{{\mathbb{C}}}

\newcommand{\bF}{{\mathbb{F}}}
\newcommand{\bM}{{\mathbb{M}}}
\newcommand{\bN}{{\mathbb{N}}}
\newcommand{\bE}{{\mathbb{E}}}

  \newcommand{\A}{{\mathcal{A}}}

  \newcommand{\E}{{\mathcal{E}}}
  \newcommand{\F}{{\mathcal{F}}}

  \newcommand{\M}{{\mathcal{M}}}
  \newcommand{\N}{{\mathcal{N}}}

\renewcommand{\S}{{\mathcal{S}}}

\renewcommand{\phi}{\varphi}
\newcommand{\upchi}{{\raise.35ex\hbox{\ensuremath{\chi}}}}
\newcommand{\eps}{\varepsilon}
\newcommand{\dd}{{\mathrm{d}}}
\renewcommand{\leq}{\leqslant}
\renewcommand{\geq}{\geqslant}

\newcommand{\AND}{\text{ and }}

\newcommand{\norm}[1]{\left\| #1 \right\|} \newcommand{\md}[1]{\left|
  #1 \right|} \newcommand{\p}[1]{\left( #1 \right)}
\newcommand{\obox}[1]{\left[ #1 \right]}
\newcommand{\set}[1]{\left\lbrace #1 \right\rbrace}
\newcommand{\dbox}[1]{\llbracket #1 \rrbracket}

\newtheorem{thm}{Theorem}[section] \newtheorem{defi}[thm]{Definition}
\newtheorem{prop}[thm]{Proposition} \newtheorem{cor}[thm]{Corollary}
\newtheorem{lemma}[thm]{Lemma} \newtheorem{rk}[thm]{Remark}

\begin{document}
\title{Sums of free variables in fully symmetric spaces} \date{}
\author{L\'eonard Cadilhac}\address{Institute of Mathematics, Polish Academy of Sciences, ul. \'Sniadeckich 8, 00-656 Warsaw, Poland}\email{leonard.cadilhac@u-psud.fr} \author{\'Eric
  Ricard} \address{ Normandie Univ, UNICAEN, CNRS, LMNO, 14000 Caen,
  France} \email{eric.ricard@unicaen.fr}

\maketitle

\begin{abstract}
We give a method to obtain, from Voiculescu's inequality, 
norm estimates for sums of free variables 
with amalgamation in general fully symmetric spaces. We use these 
estimates to interpolate the Burkholder inequalities for non commutative
martingales. The method is also applicable to other similar
settings. In that spirit, we improve known results on the
non commutative Johnson-Schechtman inequalities and recover Khinchin
inequalities associated to free groups.
\end{abstract}
\section{Introduction}

 Versions of Khinchin inequalities for Schatten classes of index
 $1<p<\infty$ were established by Lust-Piquard in the mid 80's. They
 are one of the first evidence of a new non commutative phenomenon; one
 has to deal with different notions of square functions in quantum
 analysis.  Since then, they were omnipresent in all the developments
 of non commutative analysis. They become part of the theory and are
 used to define the right function spaces. For instance the
 formulation of the Burkholder-Gundy inequalities for martingales led
 to the definition of the column Hardy space $\mathcal H_p^c(\M)$ and
 its row version $\mathcal H_p^r(\M)$ associated to a semi-finite von
 Neumann algebra. The martingale Hardy spaces are then defined as
 $\mathcal H_p(\M)=\mathcal H_p^c(\M)\cap \mathcal H_p^r(\M)$ for
 $2<p<\infty$ and $\mathcal H_p(\M)=\mathcal H_p^c(\M)+ \mathcal
 H_p^r(\M)$ for $1\leq p<2$.

 One of the main drawback is the difficulty to understand the
 behavior of those square functions with respect to interpolation
 theory. Indeed, it not clear how to deal with intersections or sums
 of two spaces in full generality. Since there is only one square
 function when the underlying algebra is commutative, those problems
 do not occur at all.  Much efforts have been made to study the
 interpolation of non  
 commutative $L_p$-inequalities in various
 contexts, for instance \cite{C, D, DPP, JRW, JS, JSZ, RW, SZ}. Given
 a function space $E$, say on $(0,\infty)$, one can associate a
 non 
 commutative space $E(\M)$ to any semi-finite von Neumann algebra
 \cite{KS}. The general question is the following: given a function
 space $E$ which is an interpolation space for $(L_p,L_q)$ and knowing an
 inequality that is true for $L_p(\M)$ and $L_q(\M)$, can we get a new
 one for $E(\M)$? Most of the papers quoted above relied on very
 elaborated machineries on function spaces and quite satisfactory
 results are available but under technical conditions (such has on Boyd
 indices, ...).

On the opposite side, freeness in quantum probabilities behaves in a
nicer way than independence in classical probabilities. The central
limit object, a semi-circular variable is bounded in $L_\infty$. Many
inequalities still hold true when $p=\infty$, the most famous example
is the Haagerup inequality for generators of free groups (which is
seen as a Khinchin type inequality). It has been a key tool to
deduce interpolation results for intersections or sums of spaces
coming from square functions \cite{P}. Very recently in \cite{C}, the
first author discovered a very efficient way to deal with
interpolation of the non commutative Khinchin inequalities. The main
novelty is that one can use certain algebraic decompositions to play
the role of square functions. With the help of freeness, it was used
to interpolate the Burkholder-Gundy inequalities. This paper is an
attempt to show that this technique is fairly general and can be used
to overcome quite easily the problems of interpolation of
non commutative function spaces, basically without any assumption.

 After the Khinchin inequality, the next interesting mixed-norms are
 given by the so called Voiculescu inequality which is a combination
 of three different norms in the spirit of the  Rosenthal inequality for
 sums of independent variables. They appeared in \cite{J1, J2} and are
 hidden in the conditioned version of the Burkholder-Gundy
 inequality. To describe them, let $\N\subset \M$ be finite von
 Neumann algebras with a normal faithful trace $\tau$ and a trace
 preserving conditional expectation $\E :\M\to \N$. For a family
 $(x_i)_{i\geq 1}$ and $1\leq p\leq \infty$, the three norms involved
 are
$$\|(x_i)\|_{p,c}= \| \big(\E \sum_{i=1}^\infty
 x_i^*x_i\big)^{1/2}\|_{L_p(\M)}, \quad \|(x_i)\|_{p,r}= \|
 (x_i^*)\|_{p,c}, \quad \|(x_i)\|_{p,d}= \Big(\sum_{i=1}^\infty \|
 x_i\|_{L_p(\M)}^p\Big)^{1/p}.$$ It was established in \cite{J1} for
 $p=\infty$ and in \cite{JPX} for $2<p<\infty$, that if the variables
 $(x_i)_{i\geq 1}$ are free and centered in $\M$ over $\N$ then for some constants
 independent of $p$
$$\| \sum_{i=1}^\infty x_i\|_{L_p(\M)}\approx
 \|(x_i)\|_{p,c}+\|(x_i)\|_{p,r}+\|(x_i)\|_{p,d}.$$ One can deduce a
 statement for $1\leq p<2$ by duality using an infimum as usual.  One
 of our objective is to interpolate those inequalities.  If $E$ is a
 symmetric space, we set
$$\|(x_i)\|_{E,c}= \| \big(\E \sum_{i=1}^\infty
 x_i^*x_i\big)^{1/2}\|_{E(\M)}, \quad \|(x_i)\|_{E,r}= \|
 (x_i^*)\|_{E,c}, \quad \|(x_i)\|_{E,d}= \big\|\sum_{i=1}^\infty
 x_i\otimes e_i\big\|_{E(\M\overline{\otimes} \ell_\infty)},$$ where $(e_i)$ is the
 canonical basis of $\ell_\infty$ equipped with its standard trace.
 We obtain that if $E$ is an interpolation space for $(L_2,L_\infty)$
 and $\sum_{i=1}^\infty x_i \in \M$,
$$\| \sum_{i=1}^\infty x_i\|_{E(\M)}\approx
 \|(x_i)\|_{E,c}+\|(x_i)\|_{E,r}+\|(x_i)\|_{E,d}.$$ Similarly when $E$
 is an interpolation space for $(L_1,L_2)$
$$\| \sum_{i=1}^\infty x_i\|_{E(\M)}\approx \inf_{x_i=a_i+b_i+d_i}
 \|(a_i)\|_{E,c}+\|(b_i)\|_{E,r}+\|(c_i)\|_{E,d},$$ one can even
 choose $a_i,\,b_i, c_i$ independently of $E$. The proof of the main
 inequalities follows three steps. First, we show that the above
 infimum is achieved when $E=L_1$ using a compactness argument. A
 sequence attaining the minimum is called an optimal decomposition. We
 then show that an optimal decomposition has a certain algebraic form
 which is used in the last step to deduce the inequalities from that
 for $E=L_\infty$. This is a fairly general principle.  In the last
 section, we sketch a proof of a similar interpolation result around
 the Haagerup-Buchholz inequality \cite{Pos, B} for words of length
 $d$ using a generating set in the free group algebra. We believe that
 the techniques could be pushed in many other directions like the
 general Rosenthal type inequality for free chaos of \cite{JPX} which
 is technically more involved; we leave it as a problem for the
 interested reader.

 We also relate our results with the Johnson-Schechtman inequalities
 for free variables obtained in \cite{SZ}.  They deal with the
 simplest case $\N=\bC$. As their commutative counterpart, they extend
 Rosenthal type inequalities to some symmetric spaces. They give a
 fully computable expression for $\| \sum_{i=1}^\infty
 x_i\|_{E(\M)}$. In this case, an algebraic decomposition can be given
 explicitly, this leads to an improvement of the constants in
 \cite{SZ}. We also explain how to use our main result to interpolate
 the conditioned Burkholder inequality in the spirit of \cite{D,
   RW}. After we submitted this article, N. Randrianantoanina informed
 us that he also obtained with Q. Xu these martingale inequalities in
 \cite{RaXu} but with different techniques.

\section{Preliminaries}

\subsection{Non commutative integration}
 We use
\cite{PX, P, terp} and \cite{FK} as general references for non
commutative integration in the semi-finite setting. We use the
classical definition of the non commutative $L_p$, $1\leq p<\infty$
associated to a semi-finite von Neumann algebra $(\M,\tau)$
$$L_p(\M,\tau)=\{ x \in L_0(\M,\tau) \;|\;
\|x\|_p^p=\tau(|x|^p)<\infty\},$$ where $L_0(\M,\tau)$ is the space of
$\tau$-measurable operators affiliated with $\M$ (see \cite{terp}). We
also set as usual $L_\infty(\M)=\M$ with its standard norm. Of course
we always have that $L_p(\M,\tau)\subset L_q(\M,\tau)$ whenever $q\leq
p$ when $\M$ is finite ($\tau(1)<\infty$).

Given $x\in L_0(\M,\tau)$, its generalized singular values \cite{FK}
is a function $\mu(x,\tau):(0,\infty)\to(0,\infty)$ which is non
increasing and has the same distribution as $x$ when $(0,\infty)$ is
equipped with the standard Lebesgue measure. We may drop the reference
to $\tau$ when it is not necessary.

 A Banach function space $(E,\|.\|_E)$ on $(0,\infty)$ is said to be
 symmetric if it is included in
 $L_0(L_\infty(0,\infty))=L_0(0,\infty)$ and if whenever $f\in E$ and
 $g\in L_0(0,\infty)$ satisfy $\mu(g)=\mu(f)$ then $g\in E$ and
 $\|g\|_E=\|f\|_E$. The non commutative version of $E$ associated to
 $(\M,\tau)$ is the space $E(\M,\tau)=\{ x\in L_0(\M,\tau) \; |\; \mu(x)\in
 E\}$ with the norm $\|x\|_{E(\M)}=\|\mu(x)\|_E$, see \cite{KS}. Note that if
 $(\N,\tau)\subset(\M,\tau)$ is a von Neumann subalgebra, then
 for $y\in L_0(\N,\tau)$, $\|y\|_{E(\N)}=\|y\|_{E(\M)}$.

 We will focus on fully symmetric function spaces. These are symmetric
 spaces $E$ such that if $f\in E$ and $g\in L_0(0,\infty)$ satisfy
 for all $t>0$, $\int_0^t \mu(g)\leq \int_0^t \mu(f)$ then $g\in
 E$ and $\|g\|_E\leq \|f\|_E$. They admit several other
 characterizations especially using general interpolation, we refer to
 \cite{BL, Kinter}. This is related to the fact that $\int_0^t
 \mu(f)=\|f\|_{L_1+tL_\infty}$ for $t>0$ and $f\in L_0(0,\infty)$.

An interpolation space for the couple $(L_p(0,\infty),L_q(0,\infty))$,
$1\leq p,q\leq \infty$ is a Banach function space $E\subset
L_p(0,\infty)+L_q(0,\infty)$ such that if
$T:L_p(0,\infty)+L_q(0,\infty)\to L_p(0,\infty)+L_q(0,\infty)$ is a
linear map such that $\|T\|_{ L_p(0,\infty)\to L_p(0,\infty)}, \|T\|_{
  L_q(0,\infty)\to L_q(0,\infty)}\leq 1$ then $T(E)\subset E$ and
$\|T\|_{E\to E}\leq 1$. Fully symmetric spaces are exactly
interpolation spaces for the couple
$(L_1(0,\infty),L_\infty(0,\infty))$.

 When $E$ is a fully symmetric function space on $(0,\infty)$ or more
 specifically an interpolation space for $(L_p,L_q)$, the non
 commutative function spaces associated to it also enjoy the same
 properties see Corollary 2.2 in \cite{PX}. We will mainly use that when
 $E$ is an interpolation for $(L_p,L_q)$ with $p\leq q$ if
 $(\M_1,\tau_{\M_1})$ and $(\M_2,\tau_{\M_2})$ are semi-finite von Neumann algebras and
 if $T: L_p(\M_1,\tau_{\M_1})\to L_p(\M_2,\tau_{\M_2})$ is a map that is
 contractive on $L_p$ and $L_q$ i.e. $\|T\|_{L_r(\M_1,\tau_{\M_1})\to
   L_r(\M_2,\tau_{\M_2})}\leq 1$ for $r=p,\,q$ then $T(E(\M_1))\subset E(\M_2)$
 and $T$ is a contraction from $E(\M_1)$ to $E(\M_2)$.
 We also refer to
 \cite{KS} for more on this topic.

 In the whole paper, we always consider fully symmetric spaces over
 $(0,\infty)$, this is not a restriction (see Remark \ref{jsmart}).
 The von Neumann algebras $(\M,\tau)$ will always be non commutative
 probability spaces ($\tau(1)=1$) except $\M\overline\otimes
 \ell_\infty$ with its natural trace and $L_\infty(0,\infty)$.

\subsection{Free products}
We fix $N\in \bN^*\cup\{\infty\}$. If $(\M_i,\tau_i)$, $1\leq i\leq N$
are finite von Neumann algebras with a common sub-von Neumann algebra
$(\N,\tau)$ and conditional expectations $\E_i:\M_i\to \N$ such that
$\tau\circ \E_i=\tau_i$, we denote by
$(\M,\tau)=\ast_{i=1...N}(\M_i,\tau_i)$ the amalgamated free product
of $(\M_i,\tau_i)$'s over $\N$.  We refer to \cite{VDN} for precise
definitions. We simply recall basic facts. If $x\in \M_i$, we denote
by $\mathring{x}=x-\E_i x$ and $\mathring \M_i=\{\mathring{x}; x\in
\M_i\}$; there is a natural decomposition $\M_i=\N\oplus \mathring
\M_i$.  The space $W=\N \oplus_{n\geq 1} \bigoplus_{1\leq i_1\neq
  i_2\neq ...\neq i_n\leq N} \mathring
  {\M_{i_1}}\otimes_{\N}...\otimes_{\N} {\mathring {\M_{i_n}}}$ is a
  $*$-algebra where the product is given by concatenation and
  centering with respect to $\N$. It has a trace given by $\tau_{\N}
  \circ \E$ where $\E$ is the natural projection onto $\N$. Then
  $(\M,\tau)$ is the finite von Neumann algebra obtained by the GNS
  construction from $(W,\tau)$. Elements in $\bigoplus_{1\leq i_1\neq
    i_2\neq ...\neq i_n\leq N}\mathring
  {\M_{i_1}}\otimes_{\N}...\otimes_{\N} {\mathring {\M_{i_n}}}$ are
  said to be of length $n$.

  To lighten notations, seeing $(\M_i,\tau_i)$ as a sub-von Neumann
  algebra of $(\M,\tau)$, we have that $\tau_{|\M_i}=\tau_i$ and
  $\E_{|\M_i}=\E_i$ and we will simply write $\tau$ and $\E$ instead
  of $\tau_i$ and $\E_i$.

 We use the notation $\mathring y$ for $y-\E y$ for $y\in L_1(\M)$.

\subsection{Column conditioned norms}

\begin{sloppypar}
In this section, we assume that $(\M,\tau)$ is a finite von Neumann
algebra with a subalgebra $\N$. The trace $\tau$ is well defined on
$L_0(\M)^+$ by \mbox{$\tau(x)=\sup_{n} \tau(x1_{[0,n]}(x))$} (see
\cite{terp}).  In particular for any $x\in L_0(\M)$,
$\tau(x^*x)=\tau(xx^*)$ and if $q_n$ is a non decreasing sequence of
projections going to 1 strongly and $x\in L_0(\M)$, then
$\tau(x^*x)=\sup_n \tau (x^*q_nx)$ (this is obvious when $x\in L_2$).
\end{sloppypar}

For $x\in L_2(\M)$, $\E x^*x$ is well defined in $L_1(\N)\subset
L_1(\M)$.

\begin{defi}
 For $x\in L_2(\M)$ and $1\leq p\leq \infty$, we set $ \| x\|_{p,c}=
 \| (\E x^*x)^{1/2}\|_p$.
\end{defi}
Note that $\|x\|_{p,c}<\infty$ when $1\leq p\leq 2$ and
$\|x\|_{2,c}=\|x\|_2$.

 The completion of the set of elements $x$ satisfying
 $\|x\|_{p,c}<\infty$ is denoted by $L_p(\M,\E)$ in \cite{J2}. Most of
 the results in this section can be collected from that paper or
 \cite{J1}. But for completeness of the next section, we will give a
 different proof of the basic facts we need; \cite{J1,J2} also deal
 with type III von Neumann algebras that we do not consider.

\begin{lemma}\label{polar}($\E$-polar decomposition) Let $E$ be right-$\N$-submodule of $L_2(\M)$.  Let $x\in E$
then there exists $u\in L_2(\M)$ such that $\|u\|_{\infty,c}\leq 1$
and $x=u(\E x^*x)^{1/2}$.

Moreover $s(\E x^*x)\leq \E u^*u\leq 1$ and $u 1_{[\eps,\infty)}(\E
  x^*x)\in E$ for all $\eps>0$.
\end{lemma}

\begin{proof}
Let $p$ be the support of $\alpha=(\E x^*x)^{1/2}$ in $\N$.
 
First note that $x=xp$ as $\tau
\big((1-p)x^*x(1-p)\big)=\tau((1-p)\alpha^2)=0$.  Then the element
$y=p\alpha^{-1}\in L_0(\N)\subset L_0(\M)$ and $u=xy$ is well defined
in $L_0(\M)$.  Clearly $u\alpha=xp=x$ in $L_0(\M)$.

Let us check that $u\in L_2(\M)$. Set $p_n=1_{\{0\}\cup [\frac 1
    n,n]}(\alpha)$, then $p_n\nearrow 1$ and by normality of $\tau$,
$$\tau(u^*u)=\tau(uu^*)=\lim_n\tau(xyp_nyx^*)=\lim_n\tau
(p_nyx^*xyp_n).$$ Since $p_ny\in \N$, by the modular property of
conditional expectations
$$\tau (p_nyx^*xyp_n)= \tau (p_ny\alpha^2 y p_n)\leq 1.$$ We also have
that
$u1_{[\eps,\infty)}(\alpha)=x1_{[\eps,\infty)}(\alpha)\alpha^{-1}\in
    E$ for all $\eps>0$ as $1_{[\eps,\infty)}(\alpha) \alpha^{-1}\in
      \N$.
  
 Next,
$$\| u\|_{\infty,c}^2= \| \E u^*u\|_\infty= \sup_{z} \tau\big(z^*(\E
 u^*u)z\big)=\sup_{z} \|uz||_2^2,$$ where $z$ runs over all elements
 in $\N$ with $\|z\|_2\leq 1$. As $u\in L_2(\M)$ and $z\in \N$,
 $\|uz\|_2=\lim_n\| u p_n z\|_2$ so that we can conclude since $\| u
 p_n z\|_2^2= \tau \big( z^*(p_nyx^*xy^*p_n)z\big)\leq \tau( z^*z)\leq
 1$.

 To get the last statement, consider similarly:
 $$\|\alpha\|_2^2=\|x\|_2^2=\tau\big(\alpha u^*u \alpha)=\lim_n
 \tau\big(u\alpha p_n \alpha u^*)= \lim_n \tau\big(p_n\alpha \E(u^* u)
 \alpha p_n)=\tau\big(\alpha \E(u^*u)\alpha \big).
 $$ The faithfulness of $\tau$ gives that $p(\E u^*u)p=p$. Since $\E
 u^*u\leq 1$, we must also have $\E(u^*u)p=p\E(u^*u)$ which is enough
 to conclude.
\end{proof}

\begin{lemma}\label{kad}
Let $u,v\in L_2(\M)$ with $\|u\|_{\infty, c},\|v\|_{\infty, c}\leq 1$,
then $\| \E(u^* v)\|_\infty\leq 1$.
\end{lemma}
\begin{proof}
As $\E$ is completely positive, we must have
$\left[ \begin{array}{cc}\E(u^*u) & \E(u^* v) \\ \E(v^*u) &
    \E(v^*v) \end{array}\right]\geq 0$. Hence one can find a
contraction $C\in \N$ so that $\E(u^* v)=\E(u^* u)^{1/2}C\E(v^*
v)^{1/2}$ from which the result follows.
\end{proof}
\begin{lemma}\label{duality} Let $E$ be a right $\N$-submodule of $L_2(\M)$ and $1\leq p\leq \infty$. For
 any $x\in E$
$$\|x\|_{p,c} =\sup_{z\in E, \|z\|_{p',c}\leq 1} |\tau (z^*x)|.$$
\end{lemma}
\begin{proof}
  Let $x=u\alpha$ be the polar decomposition given by Lemma \ref{polar}.

  First we assume that $\|x\|_{p,c}<\infty$, that is $\alpha \in L_p(\N)$.  Let $z=v\beta$ be the polar
decomposition of $z\in E$ with $\|z\|_{p',c}\leq 1$, then $\beta\in
L_{p'}(\N)$ has norm less than 1. Let $p_n=1_{\{0\}\cup [\frac 1
    n,n]}(\alpha)$ and $q_n=1_{\{0\}\cup[\frac 1 n,n]}(\beta)$, then
as $x, z\in L_2(\M)$
$$\tau(z^*x)=\lim_n \tau \big(q_n \beta v^*u \alpha
p_n\big)=\lim_n\tau \big(q_n \beta \E(v^*u) \alpha
p_n\big)=\tau\big(\beta \E(v^*u)\alpha\big),$$ thus using lemma
\ref{kad} $|\tau(z^*x)|\leq \|\alpha\|_p
\|\beta\|_{p'}\leq\|\alpha\|_p$.

To get the reverse inequality even if $\alpha\notin L_p(\N)$, it suffice to consider $z_n=up_n
\alpha^{p-1}/\|p_n\alpha\|_p^{p-1}\in E$ and let $n\to \infty$ if $1\leq p<\infty$.  For
$p=\infty$, one can take $z_n=u a_n$ where $a_n\in \N$ is of norm 1 in
$L_1(\N)$ with support in $1_{\alpha>\frac 1 n}$ and norming $\alpha$
at the limit.

\end{proof}

\section{Sums of free variables}

\subsection{Basic facts}

 A convenient way for us to look at sums of free variables is to see
 them as elements of length 1 in the free product
 $(\M,\tau)=\ast_{i=1...N}(\M_i,\tau)$.  The trace preserving
 conditional expectation from $\M$ to $\M_i$ will be denoted by
 $\E_i$.

 For $1\leq p\leq \infty$, we denote by $E_p$ the closure of the span
 of words of length 1 in $L_p(\M)$. They are $\N$-bimodules and it is
 well known that the natural orthogonal projection from $L_p(\M)$ to
 $E_p$ is bounded and of norm less than 4 (see \cite{RX} or \cite{JPX}).

 We recall that we use the notation $\|x\|_p$ for the $L_p$-norm
 without referring to the underlying algebra as it does not depend on
 it by the compatibility we impose on traces.

 Any $x\in E_2$ can be decomposed as $x=\sum_i x_i$ where $x_i\in
 \mathring L_2(\M_i)$ and with $\|x\|_2^2=\sum_i \|x_i\|^2_2$.

  The conditioned column norm we introduced in the previous section
  can be expanded in terms of $x_i$'s as $\E(x^*x)=\sum_i \E
  (x_i^*x_i)$.
\begin{defi} For $x\in E_2$ and $p\in [1,\infty]$,  we denote 
$$\| x\|_{p,c}=\Big\| \Big(\sum_i \E (x_i^*x_i) \Big)^{1/2} \Big\|_p,\qquad \|
  x\|_{p,r}=\Big\| \Big(\sum_i \E (x_ix_i^*) \Big)^{1/2} \Big\|_p,\qquad \|
  x\|_{p,d}= \big\|\sum_i x_i\otimes e_i \big\|_{ L_p(\M\overline \otimes
  \ell_\infty)},$$
where $(e_i)$ stands for the standard basis of $\ell_\infty$.
\end{defi}
Note that $\|x\|_{p,r}=\|x^*\|_{p,c}$, that allows to deduce easily
results for $\|.\|_{p,r}$ from those for $\|.\|_{p,c}$.
Of course, $\|x\|_{p,d}=\big(\sum_i \| x_i\|_p^p \big)^{1/p}$ when $1\leq p<\infty$ with the usual modification for $p=\infty$.

Viewing $E_2$ as a subspace of $L_2(\M)$, we have a notion of
$\E$-polar decomposition given by Lemma \ref{polar}. More precisely, for any
$x=\sum_i x_i\in E_2$ there is a decomposition $x=u(\E x^* x)^{1/2}$
where $u\in E_2$ satisfies $\|u\|_{\infty,c}\leq 1$ and $\E
x^*x=\sum_i \E x_i^*x_i \in L_1(\N)$. Moreover $s(\E x^*x)\leq \E
u^*u\leq 1$.

Lemma \ref{duality} can also be made more precise in our present
context and we obtain:
\begin{lemma}\label{duality2}
Let $1\leq p\leq \infty$ and $x=\sum_i x_i \in E_2$  then
$$\|x\|_{p,c} =\sup_{z\in E_2, \|z\|_{p',c}\leq 1} |\tau ( z^*x)|,$$
where $\tau(z^* x)=\sum_i \tau(z_i^*x_i)$.
\end{lemma}

We recall the Voiculescu inequality which is our fundamental tool. It
was first proved in \cite{V} when $\N=\bC$ and with amalgamation in
\cite{J1}. Any element in $x\in E_\infty\subset E_2$ can be written as
$x=\sum_i x_i$ where actually $x_i=\E_i x\in \mathring \M_i$ and the
sum converges in $L_2$. We have the following:
\begin{thm}[Voiculescu]
Let $x=\sum_i x_i \in E_\infty$, then
$$ \max\{\| x\|_{\infty,c}, \,\| x\|_{\infty,r}, \| x\|_{\infty,d}\}
\leq \|x\|_\infty \leq \| x\|_{\infty,c} +\| x\|_{\infty,r}+ \|
x\|_{\infty,d}.$$
\end{thm} 
It will be convenient to write for $x\in E_\infty$:
$$\|x\|_{\infty,\cap}=\max\{\| x\|_{\infty,c}, \,\| x\|_{\infty,r},
\| x\|_{\infty,d}\}.$$

Our main goal is to find a version of these inequalities for general
fully symmetric spaces. Using duality one obtains an estimate of 
the norm of sums of free variables in $L_1$ given by an infimum.

\subsection{Algebraic decompositions}

The heart of our argument is to obtain an algebraic
decomposition in the spirit of \cite{C} where it was done to study
Khinchin inequalities. The idea is to look for a decomposition of a
sum of free variables in $L_\infty$ that is optimal for a variant of
the dual norm of $\|.\|_{\infty, \cap}$. We first justify that such
decompositions do exist.

\begin{prop}\label{infat}
Let $x\in E_\infty$ then there exists a decomposition $x=a+d+b$ with
$a,d,b\in E_2$ such that
$$\inf_{\substack{x=\alpha+\gamma+\beta\\\alpha,\beta,\gamma\in E_2}} \| \alpha\|_{1,c}+\|
\gamma\|_{1,d}+\|\beta\|_{1,r}= \| a\|_{1,c}+\| d\|_{1,d}+\|b\|_{1,r}.$$
\end{prop}
Before going into the proof, we prove an intermediate lemma
\begin{lemma}\label{L2borne}
Let $x\in E_\infty$ such that $x=\alpha+\gamma+\beta$ with $\alpha,\gamma,\beta \in E_2$,
then there exists another decomposition $x=a+d+b$ with $a,d,b\in E_2$
such that $\|a\|_{1,c}\leq \|\alpha\|_{1,c}$, $\|b\|_{1,r}\leq
\|\beta \|_{1,r}$, $\|d\|_{1,d}\leq \|\gamma\|_{1,d}$ and $\|a\|_{\infty,c},
\|b\|_{\infty,r},\|d\|_2\leq 5\|x\|_\infty$.
\end{lemma}

\begin{proof}
Set $M=\|x\|_\infty$ and define $e=1_{[0,M]}\big((\E
\alpha^*\alpha)^{1/2}\big)\in \N$, $f=1_{[0,M]}\big((\E \beta \beta ^*)^{1/2}\big)\in
\N$. We consider the decomposition
$$x= a+d+b= (f\alpha e+x(1-e)) + f\gamma e + (f\beta e +(1-f)xe).$$ Note that
$(f\alpha e)^*(f\alpha e)\leq e(\alpha^*\alpha)e$. By the operator monotony of the
square root $$\big(\E(f\alpha e)^*(f\alpha e)\big)^{1/2} \leq \big(e(\E
\alpha^*\alpha)e\big)^{1/2}=e(\E \alpha^*\alpha)^{1/2}.$$ Since $\|.\|_{1,c}$ is a
norm by Lemma \ref{duality2}:
$$\| a\|_{1,c}\leq \| f\alpha e \|_{1,c}+ \|x(1-e)\|_{1,c}\leq \tau \big( e
(\E \alpha^*\alpha)^{1/2}\big)+ \tau \big((1-e) M \big)\leq \tau \big((\E
\alpha^*\alpha)^{1/2}\big).$$ Similarly $\|b\|_{1,r}\leq \|\beta \|_{1,r}$ and
$\|d\|_{1,d}\leq \|\gamma\|_{1,d}$ is obvious.
 
By the same argument, it is also clear that $\|a\|_{\infty, c},
\|b\|_{\infty, r}\leq 2M$ (as $\|x\|_{\infty,c}\leq \|x\|_\infty$),
hence $\|a\|_2,\|b\|_2\leq 2M$.  Thus by the triangle inequality
$\|d\|_{2}\leq 5 M$.
\end{proof}

\begin{proof}[Proof of Proposition \ref{infat}]
Take a sequence of decompositions $x=a_n+d_n+b_n$ which are optimal up
to $\frac 1n$. By Lemma \ref{L2borne} we can assume that they are
uniformly bounded in $L_2$. By taking subsequences and convex
combinations, we may obtain new sequences $(a_n)$, $(b_n)$, $(d_n)$
converging in $L_2$ to $a,b,d$ with the same properties. We must have
$x=a+d+b$. As the identity map from $(E_2,\|.\|_2)$ to
$(E_2,\|.\|_{1,c})$ is continuous (and similarly for $\|.\|_{1,r}$ and
$\|.\|_{1,d}$), we can conclude that $a,b,d$ achieve the infimum.
\end{proof}

For $x\in E_2$, we set
$$\|x\|_{1,\Sigma}=\inf_{x=a+d+b; a,b,d\in E_2} \| a\|_{1,c}+\|
d\|_{1,d}+\|b\|_{1,r}.$$

On $E_2$, we have several norms $\|.\|_{1,\bullet}$ for
$\bullet\in\{c,r,d\}$ as well as $\|.\|_1$.  The Voiculescu inequality
with Lemma \ref{duality2} gives that for $x\in E_2$, $\|x\|_1\leq
C\|x\|_{1,\bullet}$ for some constant $C$.

Assume for now that $N$ is finite (we add an exponent $N$ to
emphasize it). Since $\mathring L_1(\M_i)$ is 2-complemented in
$L_1(\M)$, $\|.\|_1$ and $\|.\|_{1,d}$ are equivalent on $E_2^N$.
From all these facts, it follows that $\|.\|_{1,\Sigma}$ is also a
norm on $E_2^N$ equivalent to $\|.\|_1$ (with constants possibly
depending on $N$). Thus the dual of $(E_2^N,\|.\|_{1,\Sigma})$ is
isomorphic to the dual of $(E_2^N,\|.\|_{1})$ which is $E_\infty^N$ as
a vector space with the usual duality bracket because $E_1$ is
complemented in $L_1$ by the orthogonal projection and $E_2$ is dense in $E_1$
for the $L_1$-norm.

Before stating the duality, we need an extra norm on $E_\infty$ given
by, for $x=\sum_i x_i\in E_\infty$, $$\|x\|_{\infty,\mathring d}= \sup_i
\inf_{e\in \N} \|x_i+e\|_\infty.$$ We clearly have
$\|x\|_{\infty,\mathring d}\leq \|x\|_{\infty,d}\leq 2\|
x\|_{\infty,\mathring d}$.

\begin{lemma}\label{duality3}
  The dual of $(E_2,\|.\|_{1,\Sigma})$ is $E_\infty$ with the norm
  $\|.\|_{\infty,\mathring \cap}=\max\{\|.\|_{\infty,
    c},\|.\|_{\infty, r}, \|.\|_{\infty, \mathring d}\}$ and
  anti-duality bracket $\langle z,x\rangle=\tau(z^*x)$ for $x\in E_2$
  and $z\in E_\infty$.
  \end{lemma}

\begin{proof}

  We have already justified the duality as vector spaces assuming $N$
  finite.  For the identification of the dual norm, let $x\in
  E_\infty^N$.  Using that $\|.\|_{1,\Sigma}$ corresponds to a sum
  norm, its dual norm is given by a supremum. By Lemma \ref{duality2},
  the dual norm on $E_\infty^N$ of $(E_2^N, \|.\|_{1,c})$ is exactly
  $\|.\|_{\infty,c}$.  The same holds for the row norm and the dual
  norm on $E_\infty^N$ of $(E_2^N, \|.\|_{1,d})$ is clearly
  $\|.\|_{\infty,\mathring d}$.
  
 Using the Voiculescu inequality once more justifies that
 $\|.\|_{1,\Sigma}$ and $\|.\|_1$ are equivalent on $E_2^N$ with some
 universal constant independent of $N$.

If we are looking at infinite free products, as $\cup_{N\in \bN^*}
E_2^N$ is dense in $E_2$ for the norm $\|.\|_2$ (bigger than both
$\|.\|_{1,\Sigma}$ and $\|.\|_1$) , $\|.\|_{1,\Sigma}$ and $\|.\|_1$
are also equivalent on $E_2$. Thus Lemma \ref{duality3} also holds for
$N=\infty$.
\end{proof}

We choose this approach to avoid looking at the dual of
$(E_2,\|.\|_{1,c})$ which may be hard to describe.

As in \cite{J1}, we will need another algebraic construction to make
the variable symmetric before finding the algebraic decomposition. We
consider the algebras $\tilde \M_i= \M_i\oplus \M_i$ with trace
$\tau((x,y))=\frac 12 (\tau(x+y))$.  Clearly $(\N,\tau)$ is identified
to a subalgebra of $(\tilde \M_i,\tau)$ by $n\mapsto (n,n)$. We simply
write $\N\subset \tilde \M_i$ not referring to the inclusion map. As
before the letter $\tau$ is used for traces on different algebras but
this is compatible with our identifications and leads to no confusion.

We consider the free product $(\tilde\M,\tau)= *_{i,\N}( \tilde
\M_i,\tau)$ with conditional expectation $\tilde \E : \tilde \M \to
\N$. Thus for $(x,y)\in \tilde \M_i$, $\tilde \E(x,y)=\frac 12 \E
(x+y)$.

The spaces of words of length one corresponding to $\tilde \M$ in $L_p$ will be
denoted by $\tilde E_p$.

For $z_i\in \M_i$, we write $\pi(z_i)=(z_i,-z_i)\in \tilde \M_i$. We
have that for $z=\sum_i z_i\in E_\infty$,
$$\E \sum_i z_i^*z_i= \tilde \E \sum_i \pi(z_i)^*\pi(z_i).$$ In
particular using the Voiculescu inequality, this allows to extend $\pi
:(E_\infty,\|.\|_\infty)\to (\tilde E_\infty\|.\|_\infty)$ by
$\pi(z)=\sum_i \pi(z_i)$ as a bounded map
with$\|z\|_{\infty,\bullet}=\|\pi(z)\|_{\infty,\bullet}$ where
$\bullet\in \{r,c,d\}$. Similarly one can extend $\pi:E_2 \to \tilde
E_2$ to an isometry for the $L_2$-norms.

The swap maps $\S_i :\tilde \M_i\to\tilde \M_i$, $\S_i(x,y)=(y,x)$ are
normal trace preserving $*$-representations which are also
$\N$-bimodular. Thus the free product $\S=*_i \S_i$ extends to a
$*$-isomorphism $\S: \tilde\M\to \tilde \M$ which is isometric on all
$L_p(\tilde \M)$. Note that $\S(\pi(x))=-\pi(x)$ for $x\in E_\infty$.

We can conclude about the algebraic decomposition:
\begin{prop}\label{decl1}
  Let $x=\sum_i x_i\in E_\infty$ there exists a sequence $(u_i)$ with
  $u_i\in \M_i$ with
$$\Big\|\Big(\E(\sum_i u_i^*u_i)\Big)^{1/2}\Big\|_{\infty},\quad
  \Big\|\Big(\E(\sum_i u_iu_i^*)\Big)^{1/2}\Big\|_{\infty},\quad \sup_i
  \|u_i\|_{\infty}\leq 1$$ and $\alpha,\beta\in \N^+$ and sequences
  $(\gamma_i),\,(\delta_i)$ in $\M_i^+$ such that $x_i=u_i\alpha+\beta
  u_i+ u_i\gamma_i$, $u_i\gamma_i=\delta_i u_i$ and
$$s(\alpha)\leq \E |u|^2\leq 1, \; s(\beta) \leq \E |u^*|^2\leq 1,\;
  s(\gamma_i)\leq |u_i|^2\leq 1,\,;s(\delta_i)\leq |u_i^*|^2\leq 1. $$
\end{prop}

\begin{rk}{\rm We point out that $\sum_i u_i^*u_i$ may only be in  $L_1(\M)^+$. Nevertheless $\E \sum_i u_i^*u_i$ sits in $\N$ (similarly for rows).}
\end{rk}

\begin{rk}\label{rela}{\rm The conditions on $u_i$ implies that $s(\gamma_i)$ and $|u_i|^2$ commute, more precisely $s(\gamma_i)=s(\gamma_i)|u_i|$.}
\end{rk}
\begin{proof}

Take $x\in E_\infty$ and consider $\pi(x)\in \tilde E_\infty$.  By
Proposition \ref{infat}, we have a decomposition $\pi(x)=a+b+d$ with
$a,d,b\in \tilde E_2$ such that $\|\pi(x)\|_{1,\Sigma}= \|
a\|_{1,c}+\| d\|_{1,d}+\|b\|_{1,r}$.  We may assume that $\S(a)=-a$ by
replacing it with $a'=\frac 12 (a-\S(a))$ as $\|a'\|_{1,c}\leq \frac
12 (\|a\|_{1,c}+\|\S (a)\|_{1,c})=\| a\|_{1,c}$. The same holds for
$b$ and $d$.

Consider the $\tilde \E$-polar decomposition $a=v \alpha$ and
$b^*=w^*\beta$ where $\alpha= (\tilde \E a^*a)^{1/2}\in L_1(\N)^+$ and
$\beta=(\tilde \E bb^*)^{1/2}\in L_1(\N)^+$ given by Lemma
\ref{polar}. We must also have $\S(v)=-v$ and $\S(w)=-w$.  Writing
$d=\sum_i d_i$, we also consider the usual polar decompositions $d_i=
z_i\gamma_i=\delta_i t_i$ with $\gamma_i,\delta_i\in L_2(\M_i)^+$ as
$d_i\in L_2(\M_i)$. We must also have that $\S(\gamma_i)=\gamma_i$,
$\S(z_i)=-z_i$ and $\S(\delta_i)=\delta_i$, $\S(t_i)=-t_i$.

 By Lemma \ref{duality3} (as $\tilde E_\infty\subset \tilde E_2$),
 there is $u'\in \tilde E_\infty$ with $\|u'\|_{\infty,\mathring
   \cap}\leq 1$ so that $\|\pi(x)\|_{1,\Sigma}=\tau
 (u'^{*}\pi(x))$. As before let $u=\frac 12 (u'-\S(u'))\in \tilde
 E_\infty$, we have that $\S(u)=-u$,
 $\tau(u^*\pi(x))=\|\pi(x)\|_{1,\Sigma}$. We clearly have that
 $\|u\|_{\infty,c}\leq \|u'\|_{\infty,c}$ and $\|u\|_{\infty,r}\leq
 \|u'\|_{\infty,r}$. For any $e\in \N$, $u_i= \frac 12 (u_i'+e -
 \S(u_i'+e))$, hence we get that $\|u\|_{\infty,d}\leq
 \|u'\|_{\infty,\mathring d}$.

From Proposition \ref{infat} we infer the equalities
$$\|\alpha\|_1 + \|\beta\|_1 + \sum_i \|
\gamma_i\|_1=\|x\|_{1,\Sigma}=\tau \Big(u^*v a+ \beta wu^*\Big)+
\sum_i \tau \big( u_i^*z_i \gamma_i\big).$$ Because of Lemma
\ref{kad}, $\|\tilde \E u^*v\|_\infty\leq 1$ and $\|\tilde \E
wu^*\|_\infty\leq 1$ and $\|u_i^*z_i\|_\infty\leq 1$. Thus, we must
have $\tau\big((\tilde \E u^*v) \alpha\big)=\|\alpha\|_1$,
$\tau\big((\tilde \E wu^*) \beta\big)=\|\beta\|_1$ and $\tau \big(
u_i^*z_i \gamma_i\big)=\|\gamma_i\|_1$ for all $i$.

Necessarily with $p=s(\alpha)$, $p=p(\tilde \E u^*v)p=\tilde \E
\big((up)^*(vp)\big)$. In particular $\tau\big((up)^*(vp)\big)=\tau
(p)\geq \|up\|_2.\|vp\|_2$.  Thus there is equality in the
Cauchy-Schwarz inequality so that $up=vp$ and $\tau \big( (\tilde \E
u^*u)p\big)=\tau (p)$.  But by Lemma \ref{polar} $a=vp\alpha$ and
$a=up\alpha=u\alpha$, and $p\leq \tilde\E uu^*\leq 1$. The same
argument also gives that $b=\beta u$ and $s(\beta) \leq \tilde\E
uu^*\leq 1$.

We also have that $\tau(u_i^*z_i\gamma_i)=\|\gamma_i\|_1$.  Thus
$s(\gamma_i)=s(\gamma_i) u_i^*z_i$ and as above $z_i=u_is(\gamma_i)$
and $s(\gamma_i)\leq |u_i|^2\leq 1$.  Similarly $d_i=\delta_iu_i$ and
$s(\delta_i)\leq |u_i^*|^2\leq 1$.

By the Voiculescu inequality $\|u\|_\infty\leq 3$ and
$\|\pi(x)\|_\infty\leq 3\|x\|_\infty$. Notice that we have $\tilde \E
u^*\pi(x)= \alpha +\E (\sum_i u_i^*\beta u_i +
u_i^*u_i\gamma_i)=\alpha +\E (\sum_i u_i^*\beta u_i+\gamma_i)$. We get
$\|\alpha\|_\infty \leq \|u^* \pi(x)\|_\infty\leq 9 \|x\|_\infty$. The
same argument works for $b$, thus $d=x-a-b\in E_\infty$ with
$\|d\|_\infty\leq 19\|x\|_\infty$. The fact that $\gamma_i\in \M_i^+$
follows from $\gamma_i=\tilde \E_i u_i^*d_i\in \M_i$.

To get the desired algebraic decomposition, it suffices to note that
as $\S(u_i)=-u_i\in \tilde \M_i$, it is of the form $u_i=(r_i,-r_i)$
for some $r_i\in \M_i$. The element $\sum_i r_i^*r_i$ makes sense in
$L_1(\M)^+$ as $\E(r_i^*r_i)=\tilde \E(u_i^*u_i)$ and thus
$\tau(r_i^*r_i)=\tau(u_i^*u_i)$. Similarly $\gamma_i=(g_i,g_i)$,
$\delta_i=(h_i,h_i)$ for some $g_i,h_i\in \M_i^+$.  Thus we get that
$x_i=r_i\alpha+\beta r_i+ r_ig_i=r_i\alpha+\beta r_i+ h_ir_i$ and the
conclusion on supports follows directly from that of $u$.
\end{proof}

\subsection{Norm estimates} We will use the algebraic decomposition to get
 estimates for sums of free variables.
\begin{thm}\label{mainineq}
With the notation of proposition \ref{decl1}, for any fully symmetric
function space $E$
$$\| x\|_{E(\M)} \leq
4\|\alpha\|_{E(\N)}+4\|\beta\|_{E(\N)}+4\Big\|\sum_i \gamma_i \otimes
e_i \Big\|_{E(\M\overline{\otimes} \ell_\infty)},$$
$$ \|\alpha\|_{E(\N)}, \|\beta\|_{E(\N)}, \Big\|\sum_i \gamma_i
\otimes e_i \Big\|_{E(\M\overline{\otimes} \ell_\infty)}\leq 4 \| x\|_{E(\M)}.$$
\end{thm}

\begin{proof}

 We fix $x=\sum_i x_i \in E_\infty$.

We start with the upper bound. We have that $\mathring u=\sum_i
\mathring u_i$ is a well defined element of $E_\infty$ by the
Voiculescu inequality as $0\leq \sum_{i=1}^N \E \mathring
u_i^*\mathring u_i\leq \sum_{i=1}^N\E u_i^* u_i$ (similarly for rows)
and $\|\mathring u_i\|_\infty \leq 2$ for all $1\leq i\leq N$. Thus
$\|\mathring u\|_\infty\leq 4$.

We point out that
$\E(x_i)=\E(u_i)\alpha+\beta\E(u_i)+\E(u_i\gamma_i)=0$. As
$a=\mathring u \alpha$, $b=\beta \mathring u$ are well defined
$d=x-a-b$ also is and necessarily $\E_i d= \mathring{(u_i\gamma_i)}$,
so we can write $x= \mathring u \alpha +\beta \mathring u + \sum_i
\mathring{(u_i\gamma_i)}$.

It is clear that $\|\mathring u \alpha \|_{E(\M)}\leq 4 \|\alpha\|_{E(\N)}$ and
similarly $\|\beta \mathring u\|_{E(\M)}\leq 4\|\beta\|_{E(\N)}$.

 Next we prove that $\big\|\sum_i \mathring{(u_i\gamma_i)}\big\|_{E(\M)} \leq
 4 \big\|\sum_i \gamma_i \otimes e_i \big\|_{E(\M\overline \otimes
   \ell_\infty)}$. Let $t>0$, we recall that for any variable in a non
 commutative measure space $\int_0^t
 \mu(z)=\|z\|_{L_1+tL_\infty}$. Consider an optimal
 decomposition $\sum_i \gamma_i\otimes e_i=r+s$ in $(L_1+tL_\infty)(\oplus \M_i)$. We have $\gamma_i=r_i+s_i$ and we may
 assume that $0\leq r_i,s_i\leq \gamma_i$.

We have $\sum_i \mathring{(u_i\gamma_i)}=\sum_i \mathring{(u_i
  r_i)}+\sum_i \mathring{(u_i s_i)}$ and $\|r\|_1=\sum_i \|r_i\|_1$,
$\|s\|_\infty=\sup_i \|s_i\|_\infty$.
 
By the Voiculescu inequality we can control $\|\sum_i \mathring{(u_i
  s_i)}\|_\infty$.  Indeed $\sup_i \|\mathring{(u_i s_i)}\|_\infty
\leq 2\sup_i \|s_i\|_\infty$. And
$$\Big\|\sum_i \mathring{(u_i s_i)} \Big\|_{\infty, c}^2 \leq
\Big\|\sum_i \E (s_i u_i^*u_is_i)\Big\|_{\infty}.$$ But $s_iu_i^*=
u_i^* u_i s_i u_i^* $ as $s(s_i)\leq s(\gamma_i) \leq u_i^* u_i\leq
1$.  Thus $0\leq\E s_i u_i^* u_is_i\leq\sup_k\{\|s_k\|_\infty^2\} \E
u_i^* u_i$ and it follows that $ \Big\|\sum_i
\mathring{\big(u_is_i\big)}\Big\|_{\infty, c}\leq
\sup_k\{\|s_k\|_\infty\} \| u\|_{\infty,c}\leq
\sup_k\{\|s_k\|_\infty\}.$

The last term is easier to handle
$$\Big\|\sum_i \mathring{\big(u_is_i)}\Big\|_{\infty, r} \leq
\Big\|\sum_i \E
(u_is_is_iu_i^*)\Big\|_{\infty}^{1/2}\leq\sup_k\{\|s_k\|_\infty\}\|u\|_{\infty,r}.$$
We obtain $\big\|\sum_i \mathring{(u_is_i)}\big\|_\infty\leq
4\sup_k\{\|s_k\|_\infty\}$.

By the triangle inequality $\big\|\sum_i \mathring{(u_ir_i)}\big\|_1\leq
2\sum_i \|r_i\|_1$. Thus we get $\big\|\sum_i
\mathring{(u_i\gamma_i)}\big\|_{L_1+tL_\infty}\leq 4 \big\|\sum
\gamma_i\otimes e_i\big\|_{L_1+tL_\infty}$ and the estimate we claimed
follows.

\bigskip

We turn to the lower bound. First, we notice that $\E (\mathring u^*
x)= \alpha +\sum_i \E(u_i^*\beta u_i+\gamma_i)$.  Indeed for any $j$,
we have $\E \mathring u_j^* x=\E \mathring u_j^* x_j= \E u_j^*x_j$
because $x=\sum_i x_i$ is centered. Thus for any $1\leq K\leq N$ $\E
\big((\sum_{i=1}^K \mathring u_i)x\big)=\E (\sum_{i=1}^K u_i^*u_i
\alpha+ u_i^*\beta u_i+\gamma_i)$. We conclude by letting $K\to
\infty$ and taking limits in $L_2$.

We get $\| \E(\mathring u^*x)\|_{E(\N)}\geq \|\alpha\|_{E(\N)}$. Since $\E$ is always a contraction, $\|\alpha\|_{E(\N)}\leq\|\mathring u^*x\|_{E(\M)}\leq 4
\|x\|_{E(\M)}$.  Similarly $\E(x\mathring u^*)=\beta +\E\Big(\sum_i
u_i\alpha_i u_i^* + \delta_i\Big)$ gives $\|\beta\|_{E(\N)}\leq 4 \|x\|_{E(\M)}$.

Fix $t>0$ and let $p=\oplus p_i\in \oplus_i \M_i$ be such that
$\|p\|_1=t, \|p\|_\infty=1$ and $\sum_i \tau(p_i\gamma_i)=\|\sum_i
\gamma_i \otimes e_i\|_{L_1+tL_\infty} $. We can also assume that
$0\leq p_i\leq s(\gamma_i)$ by replacing it by
$s(\gamma_i)(\frac{p_i+p_i^*}2)_+ s(\gamma_i)$. Thus by Remark
\ref{rela}, $p_i|u_i|=p_i=|u_i|p_i$. We use that $\gamma_i=u_i^*x_i
-|u_i|^2\alpha-u_i^*\beta u_i$.  We have $\sum_i \tau(p_i
|u_i|^2\alpha)= \sum_i \tau(|u_i|p_i |u_i|\alpha)\geq 0$ and also
$\sum_i \tau(u_ip_i u_i^*\beta)\geq 0$.

Finally, as $x_i$ is centered,
$\tau(p_iu_i^*x_i)=\tau(\mathring{(u_ip_i)}^*x)$. Let $\eps_i$ be the
argument of ${\tau(p_iu_i^*x_i)}$. We can check that $\|\sum_i
\eps_i\mathring{(u_ip_i)}\|_{\infty,c}\leq \|\sum_i \eps_i
u_ip_i\|_{\infty,c}\leq 1$; indeed again
$$\big\|\sum_i \eps_i u_ip_i \big\|_{\infty,c}=\big\| (\E \sum_i |u_i| p_i
|u_i|)^{1/2}\big\|_\infty\leq 1.$$ In the same way, $\|\sum_i
\mathring{(u_ip_i)}\|_{\infty,r}\leq 1$ and $\sup_i
\|\mathring{(u_ip_i)}\|_{\infty}\leq 2$. The element $X=\sum_i
\eps_i\mathring{(u_ip_i)}$ is well defined in $\M$ with norm less than
4 thanks to the Voiculescu inequality.  By the triangle inequality
$\|X\|_1\leq 2\sum_i \|p_i\|_1= 2t$.  Thus gathering the inequalities
$$\big\|\sum_i \gamma_i \otimes e_i\big\|_{L_1+tL_\infty}\leq \sum_i
|\tau(p_iu_i^*x_i)|=\tau \big(X^*.( \sum_i x_i)\big)\leq 4
\|x\|_{L_1+tL_\infty}.$$ As $E$ is fully symmetric, $\|\sum_i \gamma_i
\otimes e_i\|_{E(\M\overline\otimes \ell_\infty)} \leq 4 \| x\|_{E(\M)}$.
\end{proof}
For a fully symmetric space $E$, define for $x\in E_\infty$
$$\|x\|_{E,\Sigma}=\inf_{a,b,d\in E_\infty; x=a+d+b} \big\| \big(\E
a^*a\big)^{1/2}\big\|_{E(\N)}+\big\|\big(\E bb^*\big)^{1/2}\big \|_{E(\N)}+\big\|\sum_i
d_i\otimes e_i\big\|_{E(\M\overline\otimes\ell_\infty)},$$ as well as $$\|x\|_{E,\cap}=\max\{\big\| \big(\E
x^*x\big)^{1/2}\big\|_{E(\N)},\,\big\|\big(\E xx^*\big)^{1/2} \big\|_{E(\N)},\,\big\|\sum_i
x_i\otimes e_i\big\|_{E(\M\overline\otimes \ell_\infty)}\}.$$ To simplify we also write $\|x\|_{E,c}=\|
\big(\E x^*x\big)^{1/2}\|_{E(\N)}$ and similarly for $r$ and $d$.

  We can have another estimate on the algebraic decomposition.

\begin{prop}\label{majsum}
With the notation of proposition \ref{decl1}, for any fully symmetric
function space $E$, we have
$$\|\alpha\|_{E(\N)}\leq \|x\|_{E,c}, \quad \|\beta\|_{E(\N)}\leq \|x\|_{E,r},
\quad \Big\|\sum_i \gamma_i \otimes e_i \Big\|_{E(\M\overline \otimes
  \ell_\infty)}\leq \| x\|_{E,d}.$$\end{prop}
\begin{proof}
We also keep the notation of the proof of Theorem \ref{mainineq}. This
is just a variation.

We have already explained that $\E (\mathring {u}^*x)\geq \alpha\geq
0$. By Lemma \ref{kad}, there is a contraction $C\in \N$ so that $\E
(\mathring {u}^*x)= \big(\E(\mathring {u}^*\mathring {u})\big)^{1/2}C
\big(\E(x^*x)\big)^{1/2}$ and we can get $\|\alpha\|_{E(\N)}\leq \|x\|_{E,c}$. The
same works for $\beta$.

For the last bound, we modify the arguments from Theorem
\ref{mainineq}. With the same notation, one just need to notice that
$\sum_i |\tau(p_iu_i^*x_i)|=\tau \big((\sum_i \eps_ip_iu_i \otimes
e_i).( \sum_i x_i\otimes e_i)\big)$.  The element $g= \sum_i
\eps_ip_iu_i\otimes e_i$ also satisfies $\|g\|_\infty\leq 1$,
$\|g\|_1\leq t$ and we get $\|\sum_i \gamma_i \otimes
e_i\|_{L_1+tL_\infty}\leq \|\sum_i x_i \otimes
e_i\|_{L_1+tL_\infty}$. This yields the result for all fully symmetric
spaces.
 
\end{proof}

We can deduce the Rosenthal type inequalities in the spirit of
\cite{JX, JPX, DPP}.

\begin{cor}\label{maincor} For any fully symmetric space $E$ and $x\in E_\infty$, we have
$$\frac 1 {16} \|x\|_{E,\Sigma} \leq \|x\|_{E(\M)} \leq 12 \|x\|_{E,\cap}.$$
  If moreover $E$ is an $(L_1,L_2)$-interpolation space $\|x\|_{E(\M)} \leq
  2\|x\|_{E,\Sigma}.$\\ If moreover $E$ is an
  $(L_2,L_\infty)$-interpolation space $\|x\|_{E,\cap}\leq 2\|x\|_{E(\M)}.$
\end{cor}

\begin{proof}
One can take $a=\sum_i \mathring u_i \alpha$, $b=\sum_i \beta\mathring
u_i $ and $d=\sum_i \mathring{u_i\gamma_i}$. We have already shown
that $a,b,d\in E_\infty$ and we can use Theorem \ref{mainineq} to get
the lower bound as $\|a\|_{E,c}\leq \|\alpha\|_{E(\N)}$, $\|b\|_{E,r}\leq
\|\beta\|_{E(\N)}$ and $\|d\|_{E,d}\leq 2 \| \sum_i \gamma_i\otimes
e_i\|_{E(\M\overline\otimes \ell_\infty)}$.  The upper bound is direct from Propositions \ref{mainineq} and
\ref{majsum} using the triangle inequality.

We justify the remaining inequality when $E$ is an
$(L_1,L_2)$-interpolation space.

Let $a\in E_\infty$ with $\E$-polar decomposition $a=u\alpha$ and
$p=s(\alpha)\in \N^+$. The map $r\mapsto ur$ defined on $p\N p$
extends to a contraction $L_2(p\N p)\to L_2(\M)$ and also to $L_1(p\N
p)\to L_1(\M)$ as $\E(ru^*ur)^{1/2}\leq r$ for $r\in p\N^+ p$. Thus by
interpolation $\|u\alpha\|_{E(\M)}\leq \|\alpha\|_{E(\N)}=\| \big(\E
a^*a\big)^{1/2}\|_{{E(\N)}}$. The same works for $b$. For the other term, we
proceed similarly by considering the map defined on $(L_1+L_2)(\oplus
\M_i)$ by $T(\sum_i d_i\otimes e_i)=\sum_i \mathring{d_i}$. $T$ has
norm smaller than 2 on all interpolation spaces between $L_1$ and
$L_2$.

 When $E$ is an $(L_\infty,L_2)$-interpolation space, the inequality
 follows in the same way. Indeed by interpolation, for $x\in E_\infty$
 $\|\E (x^*x)\|_{L_1+tL_\infty}\leq \|x^*x\|_{L_1+tL_\infty}$ which
 implies $\|x\|_{E,c}\leq \|x\|_{E(\M)}$ by \cite{LS}; similarly for
 rows. The map $z\in \M\mapsto \sum_i (\E_i(z)-\E(z))\otimes e_i$ has
 norm less than 2 on $L_\infty$ and $L_2$.
\end{proof}
\begin{rk}{\rm Actually the inequality $\|x\|_{E(\M)}\geq c\|x\|_{E,\Sigma}$ implies
$\|x\|_{E(\M)}\leq 12c\|x\|_{E,\cap}$ for all $E$ and $x\in E_\infty$ by
    duality. Indeed, take $y\in L_\infty\cap tL_1(\M)$ norming $x$ in
    $L_1+tL_\infty$ and apply the inequality to $z=P_1(y)$, its
    component of length one, for $E=L_\infty\cap tL_1$. We get a
    decomposition $z=a+b+d$ which gives by duality that
    $\|x\|_{L_1+tL_\infty}\leq4c \|x\|_{L_1+tL_\infty,\cap}$.  This
    allows to conclude to $\|x\|_{{E(\M)}}\leq 12c \|x\|_{E,\cap}$ for all
    $E$. }

\end{rk}

\begin{rk}\label{compa}{\rm
Free or independent variables are examples of martingales. The above
corollary can be deduced from the literature when $E$ is an
interpolation space between $L_p$ and $L_p$ but only when $2<p<q<\infty$ in
\cite{D, DPP} or $1<p<q<2$ in \cite{RW} under the Fatou assumption for
$E$ (the results are given in terms of Boyd indices which is sightly
stronger).  A related estimate is obtained when $2=p<q=4$ in
\cite{JSZ} but for a somehow different norms on the right see Remark
\ref{jsmart}. The novelty is that we allow the end points of the
scales, the constants are universal and the decomposition is
independent of $E$.}
\end{rk}

\section{Applications to other inequalities}

From now on, to lighten the notation, when no confusion can occur, we may
simply write $\|x\|_E$ for $\|x\|_{E(\M)}$ if $x\in
L_0(\M,\tau)$. We do not lose information this way since
$\mu(x)$ does not depend on the
semi-finite von Neumann subalgebra in which it is computed as long as
the traces are compatible.

\subsection{Links with the Johnson-Schechtman inequalities}

The Johnson-Schechtman inequality \cite{JS} is a very efficient tool
to compute explicitly the norm of sums of independent (commutative)
variables.  Its free analogue was considered in \cite{SZ}.  It
concerns only variables with trivial amalgamation. Here we explain how
to recover it from our arguments as an algebraic decomposition can be
given explicitly.  As above $\M$ is the free product of non
commutative probability spaces $(\M_i,\tau)$, $1\leq i\leq N$, with
$\N=\bC$. This will yield much nicer constants than
the original proof. We will concentrate only on the case of symmetric
variables as in \cite{SZ} i.e. selfadjoint variables $x$ such that $x$
and $-x$ are equimeasurable.

\medskip
 
We give ourselves a fully symmetric function space $E$ on $(0,\infty)$
with $\|1_{[0,1]}\|_E=1$. With this normalization, we always have
$\|f\|_1\leq \|f\|_E\leq \|f\|_\infty$ for any $f\in L_\infty$
supported on a measure 1 set.  Define as in \cite{SZ}
$$\|f\|_{Z_E^2}= \| \mu(f) 1_{[0,1]}\|_E + \|\min\{ \mu(f)(1),
\mu(f)\}\|_2.$$

\begin{prop}
Let $x_i$ be symmetrically distributed variables in $\M_i$ then
$$\frac 1 3 \big\| \sum_i x_i\otimes e_i\big\|_{Z_E^2}\leq \big\| \sum_i
x_i\big\|_E\leq 3\big\| \sum_i x_i\otimes e_i\big\|_{Z_E^2}.$$
\end{prop}
\begin{proof}
By enlarging the algebras and using compositions with complete
isometries, we can assume that $\M_i=L_\infty[-\frac 12, \frac12]$ and
the $x_i$ are odd functions.

Assume $f=\sum_i x_i\otimes e_i \in L_1+L_\infty(\M\overline \otimes
\ell_\infty)$ with $\|f\|_{Z_E^2}<\infty$. Set $t=\mu(f)(1)$ and
$\alpha=\|\min\{ t, \mu(f)\}\|_2$, note that $\alpha\geq t$.

We can find even projections $q_i$ in $\M_i$ so that $1_{|x_i|>t}\leq
q_i\leq 1_{|x_i|\geq t}$ and $\sum_i q_i|x_i|\otimes e_i$ has
distribution $\mu(f)1_{[0,1]}$.

Considering the polar decomposition $x_i=r_i|x_i|$, we define $v_i=r_i
q_i$, $w_i=\frac 1 t {x_i} (1-q_i)$. They are disjointly supported. Let
$u_i=v_i+w_i$. Because $x_i$ is an odd function, $v_i$, $w_i$ and $u_i$ also
are.  Clearly $\|u_i\|_\infty\leq 1$ and $|\sum u_i\otimes e_i|$ has
the same distribution as $\min\{1,|f|/t\}$. Hence we have that
$\|\sum_i u_i\|_2= \alpha/t$. By the Voiculescu inequality, $\| \sum_i
u_i\|_\infty\leq 2\| \sum_i u_i\|_2 +1 \leq 3 \alpha/t$.
Similarly by construction the support of $\sum_i v_i\otimes e_i$ has
measure 1, from which we deduce that $\| \sum_i v_i\|_\infty\leq 3 $.

We have a decomposition $x_i=t u_i + v_i (|x_i|-t)_+$. We set
$\gamma_i=(|x_i|-t)_+q_i$.

We have a first trivial estimate $\| \sum_i t u_i\|_E\leq\| \sum_i t
u_i\|_\infty\leq 3\alpha$.

By construction $\|\sum \gamma_i\otimes e_i\|_E\leq \| \mu(f)
1_{[0,1]}\|_E$.  Using the same argument as in Theorem \ref{mainineq},
we get $\|\sum_i v_i\gamma_i\|_E \leq 3 \|\sum \gamma_i\otimes
e_i\|_E$ (there is no centering).

Thus we arrive at $\|\sum_i x_i \|_E\leq 3 \|\sum_i x_i\otimes e_i
\|_{Z_E^2}$.

In the opposite direction, assume $\sum x_i\in E$. Since,
$$3\alpha/t\|\sum_i x_i\|_1\geq \tau\big( (\sum_i u_i)^*(\sum_i
x_i)\big)\geq t\|\sum_i u_i\|^2=\alpha^2/t,$$ we obtain 
$\|\sum_i x_i\|_E\geq \|\sum_i x_i\|_1\geq \alpha/3$.

Taking $\sum_i p_i \otimes e_i$ with $p_i$'s even norming $\sum_i
|x_i|q_i$ in $L_1+tL_\infty$ and arguing as in Theorem \ref{mainineq},
one gets $\| \sum_i |x_i|q_i\otimes e_i\|_E \leq 3\|\sum_i
x_i\|_E$. So that we get $\| \mu(f) 1_{[0,1]}\|_E \leq 3 \|\sum_i
x_i\|_E$.
\end{proof}

\subsection{Application to martingales}

Sums of free variables are basic examples of martingales. In this
short section we explain how the norm estimates we got can be used to
interpolate the Burkholder inequality for non commutative martingales
quite easily. The basic idea is to realize the norm in the Burkholder-Gundy
inequality as the norm of a sum  of free variables with amalgamation.

Let $(\N_k)_{k=0,...d}$ be a finite filtration of the finite von
Neumann algebra $(\N,\tau)$ ($\N_d=\N$) with conditional expectations
$\bE_k$.

As usual for an element $x\in \N$, we consider its martingale
difference $(\dd x_k)_{k=0}^{d}$, where $\dd x_0=\bE_0 x$ and $\dd x_k=\bE_k
x-\bE_{k-1} x$ for $k\geq 1$.

Denote $\M_k=\N*_{\N_k}\N_{k+1}$ for $k=0,...,d-1$. The copy of $\N$ in
$\M_k$ will be still denoted by $\N$, and we denote by $\rho_k:\N_{k+1}\to
\M_k$ the natural inclusion. The conditional
expectation $\M_k\to \N_k$ onto the amalgam is denoted by $\E_k$
whereas $\E$ will denote the conditional expectation onto the 
copy of $\N$.  Thus $\E_k=\E \E_k=\E_k \E= \bE_k \E$.

We have that for any $x\in \N_{k+1}$, $\E \rho_k(x)=\E_k
\rho_{k}(x)=\E_{k} x=\bE_{k}x$.

Let $(\M,\E)=\ast_{\N}(\M_k,\E)$ and denote by $\F_{k}, \F_k^+$ the
conditional expectation from $\M$ to $\rho_k(\N_k)=\N_k\subset \M_k$ and to
$\rho_k(\N_{k+1})\subset \M_k$ for $k\geq 0$.   Hence, for $x_k\in \M_k$, we have
$\F_{k}(x_k)=\E_k(x_k)$ and $\E \F_k^+(x_k)=\E_k(x_k)$.

Define $\gamma : \N\to \M$ by $\gamma(x)=\bE_0 x+\sum_{k=1}^d
\rho_{k-1}(\dd x_k)$. And let $\phi:\M\to \gamma(\N)$ given by
$\phi(x)=\F_{0} x + \sum_{k=0}^{d-1} (\F_{k}^+-\F_{k})(x)$. Note that
$(\F_{k}^+-\F_{k})(x)$ is centered in $\M_k$ with respect to $\E$ and $\phi(x)$ is a sum of
centered free variables up to $\bE_0 x$.

\begin{prop}\label{martcp}
  The map $\phi:\M\to \M$ is a projection onto $\gamma(\N)$ and it
  extends to a bounded map $L_p(\M)\to L_p(\M)$ for $1<p<\infty$ (with
  a constant independent of $d$).
\end{prop}

We recall the dual Doob inequality from \cite{J2}.

\begin{thm}
Let $1\leq p<\infty$, there is a constant $c_p$ (only depending on
$p$) such that for any $d\geq 1$ and $a_k\in \N^+$:
$$\big\| \sum_{k=0}^{d-1} \bE_k a_k\big\|_p \leq c_p\big\| \sum_{k=0}^{d-1} a_k\big\|_p.$$
\end{thm}
\begin{proof}[Proof of Proposition \ref{martcp}]
  The first point is clear by construction. 

  To see that $\phi$ is bounded on $L_p(\M)$, we rely on the dual Doob
  inequality above. Assume first that $p\geq 2$.  The formal
  projection $P_1:L_p(\M)\to E_p\subset L_p(\M)$ onto words of length 1 has norm
  less than 4 (see \cite{RX}). Note that for $x\in \M$ if $P_1(x)=\sum_{k=0}^{d-1} x_k$ then
  $\phi(x)=\bE_0 \E x+\sum_{k=0}^{d-1} \F_k^+(x_k)$. Thus, define a map on $E_\infty$
  by $T(\sum_{k=0}^{d-1} x_k)=  \sum_{k=0}^{d-1}
  \F_{k}^+(x_{k})$ so that $\phi(x)=T(P_1(x))+\bE_0\E x$ for
  $x\in \M$. Hence we need to justify that $T$ is bounded for the
  $L_p$-norm independently of $d$.
  It suffices to check it for the three norms appearing in
  the free Rosenthal inequality.
  For the column norm, first note
  that
  $$\E |\F_{k}^+ (x_{k})|^2 \leq \E\F_{k}^+ |x_{k}|^2=
  \E_{k} |x_{k}|^2=\bE_{k}\E |x_{k}|^2.$$ From the dual Doob
  inequality in $L_{p/2}$ with $a_k=\E |x_k|^2$, we get that $T$ is bounded
  for the column norm.

  The row estimate is similar whereas the diagonal one is easy (with
  no constant depending on $p$.

  It is algebraically clear that $\phi$ is the projection onto $\gamma(\N)$.
  Noticing then that $\phi$ is an orthogonal projection on $L_2$ also gives
  that it extends to a bounded map on $L_{p'}$ with the same norm.
\end{proof}

Using our main estimate, we get
that the norms on $\N$ given $\|\gamma(x)\|_{E(\M)}$ are compatible
with interpolation.

\begin{cor}\label{mart}
  If $E$ is an interpolation space for $(L_2,L_p)$ with $2<p< \infty$
  then for $x\in \N$.
  \begin{eqnarray*}
    \| x \|_{E}&\approx_p&  \Big\| \Big(|\bE_0x|^2+ \sum_{k=1}^d
    \bE_{k-1}(|\dd x_k|^2)\Big) ^{1/2} \Big\|_E + \Big\| \Big(|\bE_0x^*|^2+
    \sum_{k=1}^d \bE_{k-1}(|\dd x_k^*|^{2})\Big) ^{1/2} \Big\|_E \\ &
    &\qquad\qquad+ \Big \|\sum_{k=0}^d \dd x_k\tens
    e_k \Big\|_{E(\N\overline\otimes\ell_\infty)} \end{eqnarray*} If $E$ is an
  interpolation space for $(L_p,L_2)$ with $1<p\leq 2$, then
  \begin{eqnarray*}
    \|x \|_{E}&\approx_p& \inf_{x=a+b+c; a,b,c\in \N}\Big\{  \Big\|
    \Big(|\bE_0 a|^2+ \sum_{k=1}^d \bE_{k-1}(|\dd a_k|^2)\Big)
    ^{1/2} \Big\|_E \\ & & \qquad \qquad + \Big\| \Big(|\bE_0b^*|^2+ \sum_{k=1}^d
    \bE_{k-1}(|\dd b_k^*|^{2})\Big) ^{1/2} \Big\|_E +  \Big\|\sum_{k=0}^d \dd
    c_k\tens e_k \Big\|_{E(\N\overline\otimes\ell_\infty)}\Big\} \end{eqnarray*}
  Moreover $a,b,c$ can be chosen to be independent of $E$ and $p$.
\end{cor}
\begin{proof}
For $\infty>p\geq 2$ by corollary \ref{maincor}, the quantity on the
right hand side is equivalent to $\|\gamma(x)\|_E$. By the Burkholder
inequality \cite{JX} $\|\gamma(x)\|_p\approx_p \| x\|_p$. Thus the map
$\gamma: \N\to \M$ and $\gamma^{-1}\phi : \M\to \N$ extend to bounded
maps on $E$ by interpolation.

 For $1< p\leq 2$, the argument is similar, we also have
 $\|x\|_E\approx_p \|\gamma(x)\|_E$ by the Burkholder inequality. Then
 one has to check that the quantity on the right hand side is also
 equivalent to $\|\gamma(x)\|_E$.  Clearly we can assume $\bE_0x=0$ by
 changing the constants. Then, if $\gamma(x)=a'+b'+c'$ with $a', b',
 c'\in E_\infty$ given by Corollary \ref{maincor}, then
 $\gamma(x)=\phi(a')+\phi(b')+\phi(c')$, and $\phi(a')=\gamma(a)$ for
 some $a\in \N$ with $\bE_0a=0$ and $\| \Big( \sum_{k=1}^d
 \bE_{k-1}(|\dd a_k|^2)\Big) ^{1/2}\|_E=\|T(a')\|_{E,c}.$ We need to
 justify that $T$ is bounded for $\|.\|_{E,c}$. First we can assume
 $\N_*$ is separable because with deal with finite families. By
 Proposition 2.8 in \cite{J2}, there are a von Neumann algebra $\tilde \M$, maps
 $u_q :\M \to \tilde \M$ for $1<q<\infty$, that are compatible in the
 sense of interpolation such that $\|u_q(x)\|_q=\|x\|_{q,c}$; moreover
 the closure of the range of $u_q$ is complemented in $L_q(\tilde \M)$.
 Thus we obtain by interpolation that $\|T(a')\|_{E,c}\leq
 c_p\|a'\|_{E,c}$ as we have shown in Proposition \ref{martcp} that
 $TP_1$ is bounded for the norm $\|.\|_{p,c}$ and $\|.\|_{2,c}$. The
 argument for $b'$ is similar and simpler for $c'$.
\end{proof}
\begin{rk}{\rm Instead of using Proposition 2.8 in \cite{J2}, 
we could have used our main estimate. Indeed, one can show for
instance that if $(\N,\tau)\subset (\M,\tau)$ and a map $T$ on $\M$ is bounded
for the norms $\|.\|_{p_i}$ and $\|.\|_{p_i,c}$ with $1\leq
p_0,p_1\leq 2$ then $T$ is bounded for the norm $\|.\|_{E,c}$ for any
$(L_{p_0},L_{p_1})$-interpolation space. This follows from a standard
limit procedure that can be found in \cite{J1}. Let $\tilde \M_n$ be the free
product over $M_n(\N)$ of $n$ copies of $M_n(\M)$ with amalgamation over
$M_n(\N)$. For $x\in \M$, denote by $\pi_{i}(x)$ its $i^{th}$
copy. One can check using Corollary \ref{maincor} that with
$\gamma_n(x)=\frac 1{\sqrt n} \sum_{j=1}^n \pi_j(x)\otimes e_{j,1}$,
we have $\lim_n \|\gamma_{n}(x)\|_{E}\approx\|x\|_{E,c}$. Moreover the
closure of the range of $\gamma_n$ is complemented in $\tilde \M_n$
(uniformly in $n$) in any fully symmetric space. This is enough to conclude, we leave the details to the 
interested reader. 
}
\end{rk}

\begin{rk}{\rm
As we already pointed it out in Remark \ref{compa}, the result was
known from \cite{D,RW} in terms of Boyd indices but one can not use
$L_2$ as an end point without extra assumptions on $E$. We also
recover as in \cite{RW} that when $p<2$, the decomposition in the
infimum can be chosen independently of $E$.  Looking carefully at the
proof shows that when $p$ goes to 1 or $\infty$ the equivalence
constants are not optimal when $E=L_p$.  It is known from \cite{JRW}
that one can actually find a decomposition independent of $p$ for
$E=L_p$ $ 1<p\leq 2$ with an optimal behavior of the constants.}
\end{rk}

\begin{rk}\label{jsmart}{\rm If we assume merely that $E$ is a fully symmetric 
function space on $(0,1)$ (with $\|1_{(0,1)}\|_E=1$), it is possible
to somehow extend $E$ to fully symmetric space on $(0,\infty)$. One
can choose for instance for $g\in L_0(0,\infty)$,
$\|g\|_{F}=\|\mu(g)1_{[0,1]}\|_E$ or $\|g\|_{Z_E^2}=\|g\|_F+ \|
g\|_{L_1+L_2}$.  It is equivalent to the definition of $Z_E^2$ above
and moreover for $f$ supported on $(0,1]$, then $\|f\|_E\approx
  \|f\|_{Z_E^2}$.  With any of theses constructions, Corollary
  \ref{mart} implies that the Johnson-Schechtman inequality in Theorem
  1.5 in \cite{JSZ} is true if $E$ is an interpolation space for
  $(L_2,L_p)$ for some $2\leq p<\infty$. It was stated only for $p=4$
  there.  Actually $Z_{E}^2$ is up to some constant an interpolation
  space for $(L_2,L_p)$ if $E$ is and it is somehow the bigger norm on
  $(0,\infty)$ extending $E$ with that property.}
\end{rk}

\subsection{Other Rosenthal inequalities}
There are many places where the norms $\|.\|_{E,\cap}$ or
$\|.\|_{E,\Sigma}$ appear. Just as in the previous section it is
possible to interpolate norm inequalities knowing the result for
$L_p$. Corollary 3.18 in \cite{MR} is one example. We simply state the
result leaving the proof to the interested reader. For an element in the
free product $(\M,\tau)=*_{\N}(\M_i,\tau)$ we write $x\in \mathcal L_i
\cap \mathcal R_i$ if $x$ is a linear combination of reduced words
starting and ending with a letter in $\mathring \M_i$, then
 \begin{cor}\label{rosfree}
  If $E$ is an interpolation space for $(L_2,L_p)$ with $2<p< \infty$
  then for any $x=\sum_i x_i\in \M$ with $x_i\in\mathcal L_i \cap \mathcal
  R_i$,
$$ \| x \|_{E} \approx_p \big\| \big(\E (x^*x)\big)^{1/2}\big\|_E +\big\| \big(\E
  (xx^* )\big)^{1/2}\big\|_E + \big\| \sum_i x_i\otimes e_i\big\|_{E(\M\overline \otimes \ell_\infty)} .$$
\end{cor}
One can also get a statement for $1<p<2$.

\subsection{Free Khinchin inequalities for words of length $d$}

In this subsection, we explain how to apply the method used above to
free Khinchin inequalities for words of fixed length $d \in
\bN^*$. The proof follows the same steps as the one of theorem
\ref{mainineq}. On one hand, the setting of Khinchin inequalities
simplifies the arguments considerably since all the variables 
are automatically centered and the duality is straightforward. On the
other hand, since we consider words of length $d$, we cannot avoid to
deal with technicalities of combinatorial nature.

Our basic tool, replacing Voiculescu's inequality, is an
inequality due to Buchholz \cite{B} and known to Haagerup. Before
stating it, let us introduce some notations. Consider the free group
with $n$ generators $\bF_{n}$ and let $\{g_i\}_{i\in \dbox{n}}$ be a
set of generators of $\bF_{n}$ with $\dbox{n} = \set{1,\dots,n}$. For
any $d\in\bN^*$ define $W_d^+ = \set{g_{i_1}\dots g_{i_d} :
  i\in\dbox{n}^d}$ and identify $W_d^+$ with $\dbox{n}^d$. Let $\M$ be
finite, $\N = \M \overline{\tens} VN(\bF_{n})$ and $\lambda : \bF_{n}
\to VN(\bF_{n})$ the natural inclusion. Let $x:W_d^+\to\M$. The
Buchholz-Haagerup inequality is expressed in terms of the matrices
$(x^{[k]}_{\alpha,\beta})_{\alpha\in W_{d-k}^+,\beta\in W_k^+}$ in
$\bM_{n^{k},n^{d-k}}(\M)$ defined by:
$$x^{[k]}_{\alpha,\beta} = x(\alpha\beta),$$ where $k$ ranges from
$0$ to $d$, see \cite{Pos}. We consider only words in the generators
(and not their inverses) to avoid cancellations and to make the
presentation easier.

It will be useful, in particular when stating lemma \ref{lem:alg}, to
embed these matrices in a common algebra $\A := \M \tens \bM_n^{\tens
  d}$ with its natural trace $\tau_\M\otimes tr^{\otimes d}$. For
$i\in \dbox{n}$, denote by $c_k$ (resp. $r_k$) the element $e_{k,1}$
(resp. $e_{1,k}$) in $\bM_n$. Define:
\begin{align*}
  i_k \colon \bM_{n^{k},n^{d-k}}(\M) &\to \A \\ a =
  (a_{i,j})_{i\in\dbox{n}^{d-k},j\in\dbox{n}^k} &\mapsto \sum_{i\in
    \dbox n^{d-k}, j\in \dbox n^{k}} a_{i,j} \tens c_{i_{1}} \tens
  \dots \tens c_{i_{d-k}}\tens r_{j_1} \tens \dots \tens r_{j_k}.
\end{align*}
Denote by $\A_k$ the range of $i_k$, which is by construction
isomorphic to $\bM_{n^{d-k},n^k}(\M)$. Define:
$$G(x) := \sum_{i\in \dbox{n}^d} x(i) \tens \lambda(g_{i_1}\dots
g_{i_d}), \AND\ [x]_k = i_k(x^{[k]}),$$ or more explicitly:
$$[x]_k := \sum_{i\in \dbox{n}^d} x(i) \tens c_{i_1} \tens \dots \tens
c_{i_{d-k}} \tens r_{i_{d-k+1}} \tens \dots \tens r_{i_d}.$$ We extend
those notations to any $X \in \A_k$ the following way: if $X = [x]_k$
for some $x\in F(W_d^+,\M)$ and $k'\in \set{0,\dots,d}$ then
$\obox{X}_{k'} := [x]_{k'}$ and $G\p{X} := G(x)$. Denote also by $tr_k
: \A \to \M\otimes \bM_n^{\tens k}$ the operator defined by $tr_k = Id
\tens tr^{\tens d-k}\tens Id^{\tens k}$.

Finally, we identify $\M\otimes \bM_n^{\tens k}$ with $\M \tens
e_{1,1}^{\tens d-k}\tens \bM_n^{\tens k}$ in $\A$; this is a non
unital trace preserving $*$-homomorphism. So that for $e\in \M\otimes
\bM_n^{\tens k}$ and $X\in \A_k$, we have $Xe\in \A_k$.

\begin{thm}[Buchholz, Haagerup]
For any $x$ in $F(W_d^+,\M)$,
$$\max_{0\leq k\leq d} \norm{[x]_k}_{\A} \leq \norm{G(x)}_{\N} \leq
\sum_{k=0}^d \norm{[x]_k}_{\A}.$$
\end{thm}

\begin{rk}\label{rem:L1}{\rm
We will also use the following dual inequality:}
\begin{equation*} 
\norm{G(x)}_{L_1(\N)} \leq \inf_{0\leq k\leq d}
\norm{[x]_k}_{L_1(\A)}.
\end{equation*}
\end{rk}
\begin{proof}
Fix $0\leq k \leq d$. Using the polar decomposition, we can decompose
$[x]_k = [y]_ka$ with $a \in \M\tens \bM_n^{\tens k}$ such that
$\norm{[x]_k}_1 = \norm{[y]_k}_2\norm{a}_2$. Write:
$$a = \sum_{j,m\in \dbox n^k} a_{j,m}\tens e_{1,1}^{\tens d-k} \tens
e_{j_1,m_1} \tens \dots \tens e_{j_k,m_k}.$$ Consider now the algebra
$\N \overline{\tens} VN(\bF_{k})$ with $(h_i)_{i\in\dbox k}$
designating a set of generators for the new copy of $\bF_{k}$. Define:
$$Y = \sum_{i\in\dbox n^{d-k},j\in \dbox n^k} y_{i,j} \tens
\lambda(g_{i_1}\dots g_{i_{d-k}}h_{j_1}\dots h_{j_k}),$$ and,
$$A = \sum_{j,m\in\dbox n^k} a_{j,m} \tens \lambda(h^{-1}_{j_k}\dots
h^{-1}_{j_{1}}g_{m_1}\dots g_{m_k}).$$ Note that if $\bE$ designates
the conditional expectation from $\A \overline{\tens}
VN(\bF_{\infty})$ to $\A$, $\bE(YA) = G(x)$. Therefore,
$\norm{G(x)}_1\leq \norm{Y}_2\norm{A}_2 = \norm{y}_2\norm{a}_2 =
\norm{[x]_k}_1$.
\end{proof}
\begin{rk}\label{normproj}{\rm Dualizing the above Remark implies that the norm of the projection in $\N$ onto words of length $d$ in the 
generators is bounded by $d+1$. This is different from the constant
$2d$ that appears in \cite{Pos} for the projection onto all words of
length $d$. Of course, this extends to any fully symmetric space using
duality and interpolation.}
\end{rk}
\begin{rk}
{\rm Since we consider only words in $W_d^+$, the Buchholz-Haagerup
  inequality can be improved. Indeed, by \cite{dlS},
$$ \norm{G(x)}_{\N} \leq 4^5\sqrt{e}\p{ \sum_{k=0}^d
    \norm{[x]_k}_{\A}^2 }^{1/2}.$$ As a consequence, the asymptotic
  behavior of constants appearing in the rest of this section can be
  made more precise. For example, the constant $d+1$ in lemma
  \ref{lem:G<A} can be replaced by $4^5\sqrt{e(d+1)}$.  }\end{rk}

To simplify some notation if $\alpha=\sum_j m_j \otimes v_j\in \M\tens
\bM_n^{\tens k}$ with $m_j\in \M$ and $v_j\in \bM_n^{\tens k}$, and
$\beta\in \bM_n^{\tens {k'-k}}$ with $k'\geq k$, we denote
by $\alpha\tens \beta\in \M\tens \bM_n^{\tens k'}$ the element $\sum_j
m_j\tens \beta\tens v_j$.

 The following lemma can be easily checked by the
reader. It is constituted of the algebraic identities that will allow
us to bypass the difficulty caused by considering words of length $d$.
\begin{lemma}\label{lem:alg}
Let $0\leq k \leq k' \leq d$. Let $a,b: \dbox{n}^d\to \M$ and $\alpha
\in \M\tens \bM_n^{\tens k}$. Then,
\begin{enumerate}[(i)]
\item $\tau_{\N}(G(a)^*G(b)) = \tau_{\A}([a]_k^*[b]_k)$,
\item $[a]_k^*[b]_k = tr_k\p{[a]_{k'}^*[b]_{k'}}$,
\item \label{iii} $[a]_k\alpha = \obox{[a]_{k'}( 1_{\bM_n}^{\tens
    k'-k}\tens \alpha)}_k$,
\item \label{iv} $[a]_{k'}(1_{\bM_n}^{\tens k'-k}\tens \alpha )=[[a]_k
  \alpha]_{k'}$.
\end{enumerate}
\end{lemma} 
We can now start to reproduce the scheme of proof of the previous
sections, starting with proving the existence of an optimal
decomposition.
\begin{prop}\label{prop:exopt}
Let $x\in F(W_d^+,\M)$. There exists $y_0,\dots,y_d \in F(W_d^+,\M)$
such that $y_0+\dots+y_d = x$ and:
$$ \sum_{k=0}^d \norm{[y_k]_k}_{L_1(\A)} = \inf_{z_0 + \dots + z_d =
  x}\set{\sum_{k=0}^d \norm{[z_k]_k}_{L_1(\A)}}.$$
\end{prop}
As usual, this proposition can be easily obtained once the following
lemma is known.
\begin{lemma}
Let $x\in F(W_d^+,\M)$. Let $y_0,\dots,y_d \in F(W_d^+,\M)$ such that
$y_0+\dots+y_d = x$. Then, there exist $z_0,\dots,z_d \in F(W_d^+,\M)$
such that $z_0+\dots+z_d = x$ and for all $0\leq k \leq d$,
$\norm{[z_k]_k}_{\infty} \leq C_{d,n} \norm{[x]_0}_{\infty}$ and
$\norm{[z_k]_k}_1 \leq \norm{[y_k]_k}_1$ for some constant $C_{d,n}$.
\end{lemma}

\begin{proof}
Note that the norms on $F(W_d^+,\M)$ given by
$\norm{[\cdot]_k}_{\infty}$ for $0\leq k\leq d$ are all
equivalent. We argue by induction, showing that for every $j\in
\set{0,\dots,d+1}$, we can find $z_0,\dots,z_d$ such that for every
$k\in \set{0,\dots,d}$, $\norm{[z_k]_k}_1 \leq \norm{[y_k]_k}_1$ and
for every $k<j$, $\norm{[z_k]_k}_{\infty} \lesssim
\norm{[x]_0}_{\infty}$. Assume that the statement is known for a fixed
$j\in \set{0,\dots,d}$ and let $w_0,\dots,w_d \in F(W_d^+,\M)$
verifying its conditions. Let $a = x - w_0 - \dots - w_{j-1}$. By the
equivalence of norms mentioned above and by our induction hypothesis,
$A := \norm{[a]_j}_{\infty} \lesssim \norm{[x]_0}_{\infty}$. Let $e =
1_{[0,A]}(\md{[w_j]_j})$ in $\M \tens \bM_n^{\tens j}$. Define
$z_0,\dots,z_d$ by:
$$ [z_k]_k =
    \begin{cases}
      [w_k]_k, & \text{if}\ k<j, \\ [w_j]_je + [a]_j(1_{\M \tens
        \bM_n^{\tens j}}-e) & \text{if}\ k=j, \\ [w_k]_k\p{
        1_{\bM_n}^{\tens k-j}\tens e } &\text{if}\ k>j.
    \end{cases}
$$ Using lemma \ref{lem:alg}, it is clear that $z_0+\dots+z_d =
    x$. The other conditions are verified in a similar way as for
    lemma \ref{L2borne}.
\end{proof}

An optimal decomposition verifies the algebraic properties expressed
in the following proposition.

\begin{prop}
Let $x\in F(W_d^+,\M)$ and $y_0,\dots,y_d$ be  an optimal decomposition of
$x$. There exists $u\in F(W_d^+,\M)$ and $\gamma_k\in \p{\M\tens
  \bM_n^{\tens k}}^+$ for all $k=0,...,d$ such that for all $0\leq
k\leq d$,
\begin{center}
$s(\gamma_k) \leq [u]_k^*[u]_k \leq 1$ in $\M\tens\bM_n^{\tens k}$ and
  $[y_k]_k = [u]_k\gamma_k$.
\end{center}
\end{prop}

\begin{proof}
We do not give a detailed proof of this proposition since it is very
similar to proposition \ref{infat} of this paper or Theorem 4.2. in
\cite{C}. The idea is to consider an optimal decomposition $x = y_0 +
\dots + y_d$ given by proposition \ref{prop:exopt} and then, to write
the polar decomposition $[y_k]_k = v_k\gamma_k$ and argue by
duality. Note that for the duality argument, it is convenient to come
back to the more standard point of view and identify $[y_k]_k \in \A_k$
with $y_k^{[k]} \in \bM_{n^{k},n^{d-k}}$.\end{proof}

\begin{rk}{\rm  Remark also that above  we consider
the right modulus of each $y_k$ even though this is not the natural
thing to do for $y_d$ which is a row. The reason is that it will make
the writing of the following proofs smoother.  }
\end{rk}

\begin{thm}\label{thm:lengthd}
Let $x\in F(W_d^+,\M)$ and $y_0,\dots,y_d$ an optimal decomposition of
$x$. Let $E$ be a fully symmetric function space. Then for all $0\leq
k\leq d$:
$$\frac 1{d+1} \norm{[y_k]_k}_{E(\A)} \leq \norm{G(x)}_{E(\N)} \leq (d+1)\sum_{k=0}^d
\norm{[y_k]}_{E(\A)}.$$
\end{thm}

\begin{lemma}\label{lem:G<A}
Let $0\leq k \leq d$, $u\in F(W_d^+,\M)$ and $\gamma \in (\M\tens
\bM_n^{\tens k})^+$. Suppose that for every $0\leq k'\leq d$,
$\norm{[u]_{k'}}_{\infty}\leq 1$ and $[u]_k^*[u]_k\gamma =
\gamma$. Then, for any fully symmetric space $E$,
$$\norm{G([u]_k\gamma)}_{E(\N)} \leq (d+1) \norm{\gamma}_{E(\A)}.$$
\end{lemma}

\begin{proof}
Since $E$ is fully symmetric, it suffices to prove the lemma for $E$
of the form $L_1 + tL_{\infty}$, $t>0$. Then, by fixing a
decomposition of $\gamma$ into two elements which can be taken to be
nonnegative, we are reduced to treating the cases of $L_1$ and
$L_{\infty}$. By Remark \ref{rem:L1}, the case of $L_1$ is
straightforward. To treat the case of $L_{\infty}$, we use the
Buchholz-Haagerup inequality. We need to prove that, for all $0\leq k'
\leq d$,
$$\norm{[[u]_k\gamma]_{k'}}_{\infty} \leq \norm{\gamma}_{\infty}.$$ Do
to this, we use lemma \ref{lem:alg}. If $k'\geq k$, we use {\it
  (\ref{iv})}:
$$\norm{[[u]_k\gamma]_{k'}}_{\infty} = \norm{[u]_{k'}(1_{\bM_n}^{\tens k'-k}
  \tens\gamma)}_{\infty} \leq \norm{[u]_{k'}}_{\infty}\norm{
  \gamma}_{\infty} \leq \norm{\gamma}_{\infty}.$$ If $k'<k$, recall
that $[u]_k^*[u]_k\gamma = \gamma$:
$$[[u]_k\gamma]_{k'}^*[[u]_k\gamma]_{k'} =
tr_{k'}(\gamma[u]_k^*[u]_k\gamma) =
tr_{k'}([u]_k^*[u]_k\gamma[u]_k^*[u]_k\gamma[u]_k^*[u]_k).$$ Now note
that $\norm{[u]_k\gamma[u]_k^*[u]_k\gamma[u]_k^*}_{\infty} \leq
\norm{\gamma^2}_{\infty}$. Hence:
$$[[u]_k\gamma]_{k'}^*[[u]_k\gamma]_{k'} \leq
\norm{\gamma^2}_{\infty}tr_{k'}([u]_k[u]_k^*) =
\norm{\gamma^2}_{\infty}[u]_{k'}^*[u]_{k'} \leq
\norm{\gamma^2}_{\infty}.$$
\end{proof}

\begin{proof}[Proof of theorem \ref{thm:lengthd}]
The second inequality is straightforward by the triangle inequality
and lemma \ref{lem:G<A} so let us focus on the first one. Once again,
it suffices to show it for $E = L_1 + tL_{\infty}$, $t>0$. We argue by
duality. Fix $0\leq k\leq d$. There exists a spectral projection $e$
of $\gamma_k$ in $\M\otimes \bM_n^{\tens k}$ such that $\tau(e) = t$
and $\tau_{\A}(e\gamma_k) = \norm{\gamma_k}_{L_1 + tL_{\infty}}$. By
lemma \ref{lem:G<A}, $\norm{G([u]_ke)}_{L_{\infty}\cap t^{-1}L_1} \leq
d+1$. Let us compute $\tau_{\N}(G([u]_ke)^*G(x))$, relying once again
heavily on lemma \ref{lem:alg} and recalling that $e[u]_k^*[u]_k = e$
and $[u]_j^*[u]_j\gamma_j = \gamma_j$ for any $j\leq d$.
\begin{align*}
\tau_{\N}(G([u]_ke)^*G(x)) &=
\sum_{j=0}^{d}\tau_{\A}(([u]_ke)^*[y_j]_k) \\ &= \sum_{j\leq k}
\tau_{\A}(e[u]^*_k\obox{[u]_{j}\gamma_{j}}_k) + \sum_{j> k}
\tau_{\A}(\obox{ e[u]_k^*}_j [u]_{j}\gamma_{j}) \\ &= \sum_{j\leq k}
\tau_{\A}(e[u]^*_k[u]_{k}(1_{\bM_n}^{\tens k-j}\tens\gamma_{j})) + \sum_{j>
  k} \tau_{\A}((1_{\bM_n}^{\tens j-k}\tens e)[u]^*_{j} [u]_{j}\gamma_{j})
\\ &= \sum_{j\leq k} \tau_{\A}(e(1_{\bM_n}^{\tens k-j}\tens\gamma_{j})) +
\sum_{j> k} \tau_{\A}((1_{\bM_n}^{\tens j-k}\tens e)\gamma_{j}) \\ &\geq
\tau_{\A}(e\gamma_k) = \norm{\gamma_k}_{L_1 + tL_{\infty}}.
\end{align*}
Hence, $(d+1)\norm{Gx}_{L_1 + tL_{\infty}} \geq \norm{\gamma_k}_{L_1 +
  tL_{\infty}}$.
\end{proof}

\begin{cor}
Let $x \in F(W_d^+,\M)$ and $E$ be a fully symmetric space. Then:
$$\dfrac{1}{(d+1)^2}\inf_{y_0+\dots +y_d =
  x}\set{\sum_{k=0}^d\norm{[y_k]_k}_E} \leq \norm{G(x)}_E \leq
(d+1)^2\max_{0\leq k\leq d}\norm{[x]_d}_E.$$ Moreover, if $E$ is an
$(L_1,L_2)$-interpolation space then,
$$\norm{G(x)}_E \leq \inf_{y_0+\dots +y_d =
  x}\set{\sum_{k=0}^d\norm{[y_k]_k}_E},$$ and if $E$ is an
$(L_2,L_{\infty})$-interpolation space then,
$$\max_{0\leq k\leq d}\norm{[x]_d}_E \leq (d+1)\norm{G(x)}_E.$$
\end{cor}

\begin{proof}

\emph{Main inequality.} The left hand side is a direct consequence of
theorem \ref{thm:lengthd}. For the right hand side, we also use
theorem \ref{thm:lengthd} together with the fact that for any $k\in
\set{0,\dots,d}$, $\norm{[y_k]_k}_E \leq \norm{[x]_k}_E$. To prove
this latter claim, we rely once again on the same technique. It
suffices to prove it for $E = L_1 + tL_{\infty}$. By duality, there
exists an element $f \in \A$ such that $\tau(f[y_k]_k) =
\norm{[y_k]_k}_{L_1 + tL_{\infty}}$, $\norm{f}_1 \leq t$ and
$\norm{f}_{\infty} \leq 1$. Furthermore, $f$ can be chosen of the form
$e[u]_k^*$ where $e$ is a projection in $\M \tens \bM_n^{\tens
  k}$. Repeating the computation above, by lemma \ref{lem:alg}:
\begin{align*}
\norm{[x]_k}_{L_1 + tL_{\infty}} \geq \tau_{\A}(e[u]_k^*[x]_k) &=
\sum_{j\leq k} \tau_{\A}(e[u]^*_k\obox{[u]_{j}\gamma_{j}}_k) +
\sum_{j> k} \tau_{\A}(\obox{e[u]_k^*}_j [u]_{j}\gamma_{j}) \\ &\geq
\tau_{\A}(e\gamma_k) = \norm{y_k}_{L_1 + tL_{\infty}}.
\end{align*}

\emph{If $E$ is an $(L_1,L_2)$-interpolation space.} First, note that
by remark \ref{rem:L1} and interpolation, for any $k\in
\set{0,\dots,d}$, $\norm{Gx}_E \leq \norm{[x]_k}_E$. Using the
triangle inequality, we obtain the desired estimate.

\emph{If $E$ is an $(L_2,L_{\infty})$-interpolation space.} This
inequality follows from the previous one by duality. The argument
necessitates to use the boundedness of the projection on words of
length $d$ given by remark \ref{normproj}.

\end{proof}

\textbf{Acknowledgement.}  The second author is supported by
ANR-19-CE40-0002.

\bibliographystyle{plain}

\bibliography{bibli}

\end{document}